\documentclass[12pt,reqno]{article}

\usepackage{amsmath,amsthm,amsfonts,amssymb,latexsym,mathrsfs,color}
\usepackage[colorlinks=true,
linkcolor=webblue, filecolor=webbrown, citecolor=webred ]{hyperref}
\definecolor{webblue}{rgb}{0,.5,0}
\definecolor{webred}{rgb}{0,.5,0}
\definecolor{webbrown}{rgb}{.6,0,0}

\newtheorem{thm}{Theorem}[section]
\newtheorem{lem}[thm]{Lemma}
\newtheorem{cor}[thm]{Corollary}
\newtheorem{prop}[thm]{Proposition}
\newtheorem{conj}[thm]{Conjecture}
\newtheorem{cl}{Claim}
\theoremstyle{definition}

\newtheorem{ex}[thm]{Example}
\newcommand{\D}{d}
\renewcommand{\d}{\delta}
\newtheorem{rem}[thm]{Remark}

\numberwithin{equation}{section}

\title{Monotonicity and log-behavior of some functions related to the Euler Gamma function
\thanks{Supported partially by the National Natural Science Foundation of China (No.
11201191) and PAPD of Jiangsu Higher Education Institutions.
\newline\hspace*{5mm}
   {\it Email addresses:} bxzhu@jsnu.edu.cn (B.-X. Zhu)}}
\author{Bao-Xuan Zhu}
\date{\footnotesize School of Mathematics and Statistics,
         Jiangsu Normal University,
         Xuzhou 221116, PR China}
\begin{document}

\maketitle

\begin{abstract}
The aim of this paper is to develop analytic techniques to deal with
certain monotonicity of combinatorial sequences. On the one hand, a
criterion for the monotonicity of the function $\sqrt[x]{f(x)}$ is
given, which is a continuous analog for one result of Wang and Zhu.
On the other hand, the log-behavior of the functions
$\theta(x)=\sqrt[x]{2 \zeta(x)\Gamma(x+1)}$ and
$F(x)=\sqrt[x]{\frac{\Gamma(ax+b+1)}{\Gamma(c x+d+1)\Gamma(e
x+f+1)}}$ is considered, where $\zeta(x)$ and $\Gamma(x)$ are the
Riemann zeta function and the Euler Gamma function, respectively. As
consequences, the strict log-concavities of the function $\theta(x)$
(a conjecture of Chen {\it et al.}) and $\{\sqrt[n]{z_n}\}$ for some
combinatorial sequences (including the Bernoulli numbers, the
Tangent numbers, the Catalan numbers, the Fuss-Catalan numbers, the
Binomial coefficients $\binom{2n}{n}$, $\binom{3n}{n}$,
$\binom{4n}{n}$, $\binom{5n}{n}$, $\binom{5n}{2n}$ ) are
demonstrated. In particular, this contains some results of Chen {\it
et al.}, Luca and St\u{a}nic\u{a}.

Finally, by researching logarithmically complete monotonicity of
some functions, the infinite log-monotonicity of the sequence
$\{\frac{(n_{0}+ia)!}{(k_0+ib)!(\overline{k_0}+i\overline{b})!}\}_{i\geq0}$
is proved. This generalizes two results of Chen {\it et al.} that
both the Catalan numbers $\frac{1}{n+1}\binom{2n}{n}$ and central
binomial coefficients $\binom{2n}{n}$ are infinitely log-monotonic
and strengths one result of Su and Wang that $\binom{dn}{\delta n}$
is log-convex in $n$ for positive integers $d>\delta$. In addition,
the asymptotically infinite log-monotonicity of derangement numbers
is showed. In order to research the stronger properties of the above
functions $\theta(x)$ and $F(x)$, the logarithmically complete
monotonicity of functions $1/\sqrt[x]{a \zeta(x+b)\Gamma(x+c)}$ and
$\sqrt[x]{\rho\prod_{i=1}^n\frac{\Gamma(x+a_i)}{\Gamma(x+b_i)}}$ is
also obtained, which generalizes the results of Lee and
Tepedelenlio\v{g}lu, Qi and Li.
\end{abstract}
\section{Introduction}
Let $\{z_n\}_{n\geq0}$ be a sequence of positive numbers. It is
called {\it log-concave} (resp. {\it log-convex}) if
$z_{n-1}z_{n+1}\le z_n^2$ (resp. $z_{n-1}z_{n+1}\ge z_n^2$) for all
$n\ge 1$. Clearly, the sequence $\{z_n\}_{n\ge 0}$ is log-concave
(resp. log-convex) if and only if the sequence
$\{z_{n+1}/z_n\}_{n\ge 0}$ is decreasing (resp. increasing).
Generally speaking, a sequence will have good behavior (e.g.,
distribution properties, bounds by inequalities) if it is
log-concave or log-convex. In addition, sequences with log-behaviour
arise often in combinatorics, algebra, geometry, analysis,
probability and statistics and have been extensively investigated
(see \cite{Bre94,LW07, Sta89,WY07,Zhu12} for instance).

Motivated by a series of conjectures of Sun~\cite{Sun-conj} about
monotonicity of sequences of the forms $\{\sqrt[n]{z_n}\}$ and
$\{\sqrt[n+1]{z_{n+1}}/\sqrt[n]{z_n}\}$, where $\{z_n\}_{n\ge 0}$ is
a familiar number-theoretic or combinatorial sequence, e.g., the
Bernoulli numbers, the Fibonacci numbers, the derangements numbers,
the Tangent numbers, the Euler numbers, the Schr\"oder numbers, the
Motzkin numbers, the Domb numbers, and so on. These conjectures have
recently been investigated by some researchers (see Chen {\it et
al.}~\cite{CGW,CGW1}, Hou {\it et al.}~\cite{HSW12}, Luca and
St\u{a}nic\u{a}~\cite{LS12}, Wang and Zhu~\cite{WZ13}). The main aim
of this paper is to develop  some analytic techniques to deal with
monotonicity of $\{\sqrt[n]{z_n}\}$ and
$\{\sqrt[n+1]{z_{n+1}}/\sqrt[n]{z_n}\}$ (Note that the monotonicity
of $\{\sqrt[n+1]{z_{n+1}}/\sqrt[n]{z_n}\}$ equals to the
log-behavior of $\{\sqrt[n]{z_n}\}$).

 Recently, Wang and
Zhu~\cite{WZ13} observed sufficient conditions that the
log-behaviour of $\{z_n\}_{n\ge 0}$ implies the monotonicity that of
$\{\sqrt[n]{z_n}\}_{n\ge 1}$. For example, for a positive log-convex
sequence $\{z_n\}_{n\ge 0}$, if $z_0\le 1$, then the sequence
$\{\sqrt[n]{z_n}\}_{n\ge 1}$ is increasing. Using the analytic
approach of Chen {\it et al.} \cite{CGW}, the following continuous
analog can be proved, whose proof is arranged in Section 2.
\begin{thm}\label{thm+inc}
Let $N$ be a positive number. If $f(x)$ is a positive increasing
log-convex function for $x\geq N$ and $f(N)\leq 1$, then
$\sqrt[x]{f(x)}$ is strictly increasing on $(N,\infty)$.
\end{thm}
\begin{rem}
Theorem~\ref{thm+inc} can be applied to the monotonicity of
$\{\sqrt[n]{z_n}\}_{n\ge 1}$ for some combinatorial sequences
$\{z_n\}_{n\ge 0}$. Some further examples and applications related
to Theorem~\ref{thm+inc} can be found in \cite{CGW}.
\end{rem}

Thus, one may ask whether there are some analytic techniques to deal
with the log-behavior of $\{\sqrt[n]{z_n}\}_{n\geq1}$. This is
another motivation of this paper. In particular, the following
conjecture of Chen {\it et al.}~\cite{CGW} is still open.
\begin{ex}
Recall that the classical Bernoulli numbers are defined by
$$B_0=1,\quad \sum_{k=0}^n\binom{n+1}{k}B_k=0,\qquad n=1,2,\ldots.$$
It is well known that $B_{2n+1}=0, (-1)^{n-1}B_{2n}>0$ for $n\ge 1$
and
$$(-1)^{n-1}B_{2n}=\frac{2(2n)!\zeta(2n)}{(2\pi)^{2n}},$$
see \cite[(6.89)]{GKP94} for instance. In order to show that
$\{\sqrt[n]{(-1)^{n-1}B_{2n}}\}$ is increasing, Chen et al
\cite{CGW} introduced the function $\theta(x)=\sqrt[x]{2
\zeta(x)\Gamma(x+1)},$ where
$$\zeta(x)=\sum_{n\geq1}\frac{1}{n^x}$$ is the Riemann zeta function and $\Gamma(x)$ is the Euler
Gamma function. Thus
$$\sqrt[n]{(-1)^{n-1}B_{2n}}={\theta^2(2n)}/{4\pi^{2}}.$$ They proved
that $\theta(x)$ is increasing on $(6,\infty)$. In addition, in
order to get the log-concavity of
$\{\sqrt[n]{(-1)^{n-1}B_{2n}}\}_{n\ge 1}$, they further conjectured.
\begin{conj}\emph{\cite{CGW}}\label{conj}
The function $\theta(x)=\sqrt[x]{2 \zeta(x)\Gamma(x+1)}$ is
log-concave on $(6,\infty)$.
\end{conj}
\end{ex}
Using some inequalities of the Riemann zeta function and the Euler
Gamma function, in Section 3, this conjecture will almost be
confirmed, see Theorem~\ref{thm+con}. As applications, the results
of Luca and St\u{a}nic\u{a}~\cite{LS12} on strict log-concavities of
$\{\sqrt[n]{(-1)^{n-1}B_{2n}}\}_{n\ge 1}$ and
$\{\sqrt[n]{T(n)}\}_{n\ge 1}$ can be verified, where $T(n)$ are the
Tangent numbers.

In addition, motivated by the strict log-concavities of
$\sqrt[n]{\binom{2n}{n}}$ and $\sqrt[n]{\frac{1}{2n+1}
\binom{2n}{n}}$ (Chen {\it et al.} \cite{CGW1}), the log-behavior of
the function
$$F(x)=\sqrt[x]{\frac{\Gamma(ax+b+1)}{\Gamma(c x+d+1)\Gamma(e
x+f+1)}}$$ is considered (see Theorem~\ref{thm+sum}). As
consequences, for any positive integers $p\geq2$ and $a> c$, the
strict log-concavities of
$\{\sqrt[n]{\frac{1}{(p-1)n+1}\binom{pn}{n}}\}_{n\geq2}$ and
$\{\sqrt[n]{\binom{an}{cn}}\}_{n\geq30}$ are obtained, see
Corollary~\ref{cor+bino}. For more examples, the sequences
$\{\sqrt[n]{\frac{1}{2n+1}\binom{2n}{n}}\}_{n\geq1}$,
$\{\sqrt[n]{\binom{2n}{n}}\}_{n\geq1}$,
$\{\sqrt[n]{\binom{3n}{n}}\}_{n\geq1}$,
$\{\sqrt[n]{\binom{4n}{n}}\}_{n\geq1}$,
$\{\sqrt[n]{\binom{5n}{n}}\}_{n\geq1}$ and
$\{\sqrt[n]{\binom{5n}{2n}}\}_{n\geq1}$ are strictly log-concave,
respectively.

 To study the conjectures
of Sun on the monotonicity of
$\{\sqrt[n+1]{z_{n+1}}/\sqrt[n]{z_n}\}$, Chen {\it et
al.}~\cite{CGW1} found a connection between the log-behavior of
$\{\sqrt[n]{z_n}\}_{n\ge 1}$ and that of $\{{z_{n+1}}/{z_n}\}_{n\ge
0}$. Moreover, they introduced a stronger concept as follows: define
an operator $R$ on a sequence $\{z_n\}_{n\geq 0}$ by
$$R\{z_n\}_{n\geq 0}=\{x_n\}_{n\geq 0},$$ where $x_n = z_{n+1}/{z_n}$. The sequence
$\{z_n\}_{n\geq 0}$ is called {\it infinitely log-monotonic} if the
sequence $R^r\{z_n\}_{n\geq 0}$ is log-concave for all positive odd
$r$ and is log-convex for all nonnegative even $r$. In fact, the
infinite log-monotonicity is related to the logarithmically
completely monotonic function.

Recall that a function $f(x)$ is said to be {\it completely
monotonic} on an interval $I$ if $f(x)$ has derivatives of all
orders on $I$ which alternate successively in sign, that is,
$$(-1)^nf^{(n)}(x)\geq0$$
for all $x\in I$ and for all $n\geq 0$. If inequality is strict for
all $x\in I$ and for all $n\geq 0$, then $f(x)$ is said to be
strictly completely monotonic. A positive function $f(x)$ is said to
be {\it logarithmically completely monotonic} on an interval $I$ if
$\log f(x)$ satisfies
$$(-1)^n[\log f(x)]^{n}\geq0$$
for all $x\in I$ and for all $n\geq 1$. A logarithmically completely
monotonic function is completely monotonic, but not vice versa, see
Berg~\cite{Ber04}. The reader can refer to \cite{Widd46} for the
properties of completely monotonic functions and \cite{QL13} for a
survey of logarithmically completely monotonic functions. In
\cite{CGW1}, Chen {\it et al.} found the link between
logarithmically completely monotonic functions and infinite
log-monotonicity of combinatorial sequences. Thus, in Section 4, the
logarithmically complete monotonicity of some functions related to
the combinatorial sequences will be considered. As applications, for
nonnegative integers $n_{0},k_{0},\overline{k_0}$ and  positive
integers $a,b,\overline{b}$, if $a\geq b+\overline{b}$ and $-1\leq
k_0-(n_0+1)b/a\leq0$, then the sequence
$$\{\frac{(n_{0}+ia)!}{(k_0+ib)!(\overline{k_0}+i\overline{b})!}\}_{i\geq0}$$
is infinitely log-monotonic. This generalizes two results of Chen
{\it et al.} \cite{CGW1} that both the Catalan numbers
$\frac{1}{n+1}\binom{2n}{n}$ and central binomial coefficients
$\binom{2n}{n}$ are infinitely log-monotonic and strengths one
result of Su and Wang~\cite{SW08} that $\binom{dn}{\delta n}$ is
log-convex in $n$ for positive integers $d>\delta$. In addition, the
asymptotically infinite log-monotonicity of derangement numbers is
also demonstreted.

 In order to research the stronger properties of the above
functions $\theta(x)$ and $F(x)$, the logarithmically complete
monotonicity of functions $1/\sqrt[x]{a \zeta(x+b)\Gamma(x+c)}$ and
$\sqrt[x]{\rho\prod_{i=1}^n\frac{\Gamma(x+a_i)}{\Gamma(x+b_i)}}$ is
also given, which generalizes one result of Lee and
Tepedelenlio\v{g}lu about the logarithmically complete monotonicity
of $\sqrt[x]{\frac{2\sqrt{\pi}\Gamma(x+1)}{\Gamma(x+1/2)}}$,  and
one result of Qi and Li about the logarithmically complete
monotonicity of $\sqrt[x]{\frac{a\Gamma(x+b)}{\Gamma(x+c)}}$.

\section{Analytic results for the
monotonicity of the sequence $\sqrt[n]{z_n}$}

This section is to give the proof of the analytic result Theorem
\ref{thm+inc}.
\begin{proof}\quad Let $y=\sqrt[x]{f(x)}$. Then one can get
\begin{equation*}
y'=\frac{y}{x}\left(\frac{f'(x)}{f(x)}-\frac{\log{f(x)}}{x}\right).
\end{equation*}
In order to show that $\sqrt[x]{f(x)}$ is strictly increasing, it
suffices to prove
\begin{equation}\label{e1}
\frac{f'(x)}{f(x)}-\frac{\log{f(x)}}{x}>0
\end{equation}
for $x\geq N$. Since $f(N)\leq 1$ and $f(x)$ is increasing, one can
derive that
\begin{equation}\label{e2}
\frac{\log{f(x)}}{x}\leq\frac{\log{f(x)}-\log{f(N)}}{x}<\frac{\log{f(x)}-\log{f(N)}}{x-N}
\end{equation}
for $x\geq N$.

 By the mean value theorem, one can obtain
\begin{equation}\label{e3}
\frac{\log{f(x)}-\log{f(N)}}{x-N}=\frac{f'(\xi)}{f(\xi)},
\end{equation}
where $N\leq\xi\leq x$. On the other hand, it follows from
log-convexity of the function $f(x)$ that
\begin{equation}
\left(\log{f(x)}\right)''=\left(\frac{f'(x)}{f(x)}\right)'=\frac{f''(x)f(x)-f'(x)^2}{f^2(x)}\geq0,
\end{equation}
which implies that $\frac{f'(x)}{f(x)}$ is increasing. Thus,  it
follows that
\begin{equation}\label{e4}
\frac{f'(\xi)}{f(\xi)}\leq\frac{f'(x)}{f(x)}
\end{equation}
for $x\geq\xi.$ Combining (\ref{e2}), (\ref{e3}) and (\ref{e4}), one
can obtain (\ref{e1}). So $\sqrt[x]{f(x)}$ is increasing.
\end{proof}

\section{Analytic results for the
log-behavior of the sequence $\sqrt[n]{z_n}$}

In order to deal with the log-behavior of the sequence
$\sqrt[n]{z_n}$, some analytic methods will be developed in this
section. There are two main results in this section, one being the
proof of Conjecture \ref{conj} and the other being the log-behavior
of the function $F(x)$.

In the proofs, the following some known facts are needed. It follows
from \cite[Theorem 8]{Al97} that the function
$$G_0(x)=-\log{\Gamma(x)}+(x-1/2)\log{x}-x+\log{\sqrt{2\pi}}+\frac{1}{12x}$$
is strictly completely monotonic on $(0, \infty)$. This implies that
\begin{eqnarray}
\log{\Gamma(x)}&<&(x-1/2)\log{x}-x+\log{\sqrt{2\pi}}+\frac{1}{12x},\label{gama+1}\\
\left(\log{\Gamma(x)}\right)'&>&\log{x}-\frac{1}{2x}-\frac{1}{12x^2},\label{gama+2}\\
\left(\log{\Gamma(x)}\right)''&<&\frac{1}{x}+\frac{1}{2x^2}+\frac{1}{6x^3}.\label{gama+3}
\end{eqnarray}
On the other hand, \cite[Theorem 8]{Al97} also says that the
function
$$F_0(x)=\log{\Gamma(x)}-(x-1/2)\log{x}+x-\log{\sqrt{2\pi}}$$
is strictly completely monotonic on $(0, \infty)$. So
\begin{eqnarray}
\log{\Gamma(x)}&>&(x-1/2)\log{x}-x+\log{\sqrt{2\pi}},\label{gama+4}\\
\left(\log{\Gamma(x)}\right)'&<&\log{x}-\frac{1}{2x},\label{gama+5}\\
\left(\log{\Gamma(x)}\right)''&>&\frac{1}{x}+\frac{1}{2x^2}.\label{gama+6}
\end{eqnarray}
Thus, by combining these inequalities, one can get the next result,
which will be used repeatedly in the proofs.
\begin{lem}\label{bound+gama}
Let $a>0$. Assume that $h(x)=\log{\Gamma(x)}$. If $b\geq -1$ and
$ax+b\geq0$, then
\begin{eqnarray*}
x^3\left(\frac{h(ax+b+1)}{x}\right)''&\leq&-ax+(2b+1)\log{(ax+b+1)}-3b-\frac{3}{2}+\log{2\pi}+\frac{b^2+b+1/2}{ax+b+1};\\
x^3\left(\frac{h(ax+b+1)}{x}\right)''&\geq&-ax+(2b+1)\log{(ax+b+1)}-3b-3+\log{2\pi}.
\end{eqnarray*}
\end{lem}
\begin{proof}
By $h(x)=\log{\Gamma(x)}$, it is not hard to deduce that
\begin{eqnarray*}
\left(\frac{h(ax+b+1)}{x}\right)''=\frac{a^2x^2h''(ax+b+1)-2axh'(ax+b+1)+2h(ax+b+1)}{x^3}.
\end{eqnarray*}
By (\ref{gama+1}), (\ref{gama+2}) and (\ref{gama+3}), it follows
that
\begin{eqnarray*}
&&a^2x^2h''(ax+b+1)-2axh'(ax+b+1)+2h(ax+b+1)\\
&\leq&-ax+(2b+1)\log{(ax+b+1)}-3b-\frac{3}{2}+\log{2\pi}+\frac{b^2+b+1/2}{ax+b+1}.
\end{eqnarray*}
In addition, by (\ref{gama+4}), (\ref{gama+5}) and (\ref{gama+6}),
one can also obtain that
\begin{eqnarray*}
&&a^2x^2h''(ax+b+1)-2axh'(ax+b+1)+2h(ax+b+1)\\
&\geq&-ax+(2b+1)\log{(ax+b+1)}-3b-3+\log{2\pi}.
\end{eqnarray*}
This completes the proof.
\end{proof}
In order to prove Conjecture \ref{conj}, the next result will be
used.
\begin{lem}\label{bound+zeta}
Let $\zeta(x)=\sum_{n\geq1}\frac{1}{n^x}$ the Riemann zeta function.
Define a function $\eta(x)=\zeta(x)-1$. The bound $\eta(x)\leq
3/2^x$ holds for all $x\geq 4$.
\end{lem}
\begin{proof}
Since
\begin{eqnarray*}
\eta(x)&=&\frac{1}{2^x}\left(1+\frac{1}{1.5^x}+\frac{1}{2^x}+\cdots\right)\\
&\leq&\frac{1}{2^x}\left(1+\frac{1}{1.5^x}+2\left(\zeta(x)-1\right)\right)\\
&\leq&\frac{1}{2^x}\left(1+\frac{1}{2}+2\eta(x)\right)
\end{eqnarray*}
for $x\geq4$, one can get $\eta(x)\leq\frac{3}{2^x}.$
\end{proof}
Now a result for Conjecture \ref{conj} can be stated as follows.
\begin{thm}\label{thm+con}
The function
$$\theta(x)=\sqrt[x]{2 \zeta(x)\Gamma(x+1)}$$
is log-concave on $(7.1,\infty)$.
\end{thm}
\begin{proof}
In order to show that $\theta(x)$ is log-concave on $(7.1,\infty)$,
it suffices to prove
\begin{eqnarray}\label{equ+thta}
\left(\log{\theta(x)}\right)''&=&\left(\frac{\log2}{x}\right)''+\left(\frac{\log
\zeta(x)}{x}\right)'' +\left(\frac{\log\Gamma(x+1)}{x}\right)''\nonumber\\
&=&\frac{2\log2}{x^3}+\left(\frac{\log \zeta(x)}{x}\right)''
+\left(\frac{\log\Gamma(x+1)}{x}\right)''\\
&<&0\nonumber.
\end{eqnarray}
Noting that $\log x<\sqrt{x}$ for $x\geq2$, one has
$\zeta''(x)<\eta(x-1)$ and $|\zeta'(x)|<\eta(x-0.5).$ In addition,
it follows from $\log(x+1)\leq x$ for $x>0$ that
$\log(1+\eta(x))\leq\eta(x)\leq\frac{3}{2^x}$ by
Lemma~\ref{bound+zeta}. Thus, for $x\geq7.1$, it follows that
\begin{eqnarray}\label{bound+zeta+d}
x^3\left(\frac{\log
\zeta(x)}{x}\right)''&=&x^2\left(\frac{\zeta(x)\zeta''(x)
-\zeta'(x)^2}{\zeta(x)^2}\right)-2x\frac{\zeta'(x)}{\zeta(x)}+2\log\zeta(x)\nonumber\\
&<&\frac{x^2\zeta''(x)}{\zeta(x)}-\frac{2x\zeta'(x)}{\zeta(x)}+2\log\zeta(x)\nonumber\\
&<&2.67 ,
\end{eqnarray}
where the final inequality can be obtained  by considering the
monotonicity of the right function.

 On the other hand, by
Lemma~\ref{bound+gama}, one can get
\begin{eqnarray}\label{bound+g}
x^3\left(\frac{\log\Gamma(x+1)}{x}\right)''&\leq&-x+\log{(x+1)}-1+\log{2\pi}+\frac{1}{2(x+1)}\nonumber\\
&<&-4.1
\end{eqnarray}
for $x\geq7.1$.

 Thus, combining (\ref{equ+thta}), (\ref{bound+zeta+d}) and
 (\ref{bound+g}), one can conclude
 \begin{eqnarray*}
\left(\log{\theta(x)}\right)''
&=&\frac{2\log2}{x^3}+\left(\frac{\log \zeta(x)}{x}\right)''
+\left(\frac{\log\Gamma(x+1)}{x}\right)''\\
&<&0,
\end{eqnarray*}
as desired. This completes the proof.
\end{proof}
Notice that
$$\sqrt[n]{(-1)^{n-1}B_{2n}}=\frac{\theta^2(2n)}{4\pi^{2}}.$$
Thus, it follows from the strict log-concavity of
$\{\theta(2n)\}_{n\geq4}$ that
$\{\sqrt[n]{(-1)^{n-1}B_{2n}}\}_{n\geq4}$ is strictly log-concave.
In addition, it is easy to check that
$\{\sqrt[n]{(-1)^{n-1}B_{2n}}\}_{n\geq0}$ is strictly log-concave
for $1\leq n\leq 4$. Thus, the following result is immediate, which
was conjectured by Sun~\cite[Conjecture 2.15]{Sun-conj} and has been
verified by Luca and St\u{a}nic\u{a}~\cite{LS12} and Chen {\it et
al.} \cite{CGW1} by different methods, respectively.
\begin{cor}
The sequence $\{\sqrt[n]{(-1)^{n-1}B_{2n}}\}_{n\ge 1}$ is strictly
log-concave.
\end{cor}
Now consider the tangent numbers
$$\{T(n)\}_{n\ge 0}=\{1, 2, 16, 272, 7936, 353792,\ldots\},\qquad \cite[A000182]{OEIS}$$
which are defined by
$$\tan{ x}=\sum_{n\geq1}T(n)\frac{x^{2n-1}}{(2n-1)!}$$
and are closely related to the Bernoulli numbers:
$$T(n)=(-1)^{n-1}B_{2n}\frac{(4^{n}-1)}{2n}4^{n},$$
see \cite[(6.93)]{GKP94} for instance. So
$$\sqrt[n]{T(n)}=4\sqrt[n]{(-1)^{n-1}B_{2n}}\sqrt[n]{4^{n}-1}\sqrt[n]{\frac{1}{2n}}.$$

It is not difficult to verify that both $\sqrt[n]{4^{n}-1}$ and
$\sqrt[n]{\frac{1}{2n}}$ are log-concave in $n$ (we leave the
details to the reader). The product of log-concave sequences is
still log-concave. So the next result is immediate, which was
conjectured by Sun~\cite[Conjecture 3.5]{Sun-conj} and was verified
by Luca and St\u{a}nic\u{a}~\cite{LS12} by a discrete method.
\begin{cor}
The sequence $\{\sqrt[n]{T(n)}\}_{n\ge 1}$ is strictly log-concave.
\end{cor}
In order to develop analytic techniques to deal with the
log-behavior of $\{\sqrt[n]{z_n}\}$. In the following, the
log-behavior of a function $F(x)$ related to the Euler Gamma
function will be considered, which can be applied to some
interesting binomial coefficients.
\begin{thm}\label{thm+sum}
Given real numbers $b,d,f$ and nonnegative real numbers $a,c,e$,
define the function
$$F(x)=\sqrt[x]{\frac{\Gamma(ax+b+1)}{\Gamma(c x+d+1)\Gamma(e x+f+1)}}.$$
\begin{itemize}
\item [\rm(i)]
If $a> c+e$, then $F(x)$ is an asymptotically log-concave function.
\item [\rm(ii)] Assume $a= c+e$. If $c\ge e>0$ and $b< d+f+1/2$, then $F(x)$ is an
asymptotically log-concave function.  In particular, if $c\ge 1$ and
$b=d=f=0$, then we have $F(x)$ is a log-concave function for $x\geq
30$; if $c\ge 1$, $b=d=0$ and $f\geq1$, then $F(x)$ is a log-concave
function for $x\geq 2$.
\item [\rm(iii)] Assume $a= c+e$. If $c> e=0$ and $b< d$, then $F(x)$ is an
asymptotically log-concave function.
\item [\rm(iv)]
If $a<c+e$, then $F(x)$ is an asymptotically log-convex function.
\end{itemize}
\end{thm}
\begin{proof}
Let $h(x)=\log{\Gamma(x)}$. By Lemma~\ref{bound+gama}, one has
\begin{eqnarray}\label{ineq+F}
\left(\log
F(x)\right)''&=&\left(\frac{h(ax+b+1)}{x}\right)''-\left(\frac{h(cx+d+1)}{x}\right)''-\left(\frac{h(e x+f+1)}{x}\right)''\nonumber\\
&=&(c+e-a)x+\log\frac{(ax+b+1)^{(2b+1)}}{(cx+d+1)^{(2d+1)}(ex+f+1)^{(2f+1)}}+3(d+f-b)\nonumber\\
&&+\frac{9}{2}-\log 2\pi+\frac{b^2+b+1/2}{ax+b+1}.
\end{eqnarray}
It is easy to prove for $a> c+e$ that
$$\lim_{x\rightarrow+\infty}(c+e-a)x+\log\frac{(ax+b+1)^{(2b+1)}}{(cx+d+1)^{(2d+1)}(ex+f+1)^{(2f+1)}}=-\infty,$$
and for $a= c+e$ that
$$\lim_{x\rightarrow+\infty}\log\frac{(ax+b+1)^{(2b+1)}}{(cx+d+1)^{(2d+1)}(ex+f+1)^{(2f+1)}}=-\infty$$
if $c\ge e>0$ and $b< d+f+1/2$ or $c> e=0$ and $b< d$.
 Thus, under conditions of (i), (ii) and (iii), respectively, by (\ref{ineq+F}) one can get $$\lim_{x\rightarrow+\infty}\left(\log
F(x)\right)''=-\infty,$$ implying that $F(x)$ is an asymptotically
log-concave function.

Assume that $a=c+e$ and $c\ge e\ge 1$. If $b=d=f=0$, then, by
(\ref{ineq+F}),
\begin{eqnarray*} \left(\log F(x)\right)''
&<&\log\frac{(ax+1)}{(cx+1)(ex+1)}+\frac{9}{2}-\log 2\pi+\frac{1}{2(ax+1)}\\
&<&-0.04
\end{eqnarray*}
for $x\geq 30$.
 If $b=d=0$ and $f\geq1$, then, by
(\ref{ineq+F}),
\begin{eqnarray*} \left(\log F(x)\right)''
&<&\log\frac{(ax+1)}{(cx+1)(ex+2)^3}+\frac{9}{2}-\log 2\pi+\frac{1}{2(ax+1)}\\
&<&-0.37
\end{eqnarray*}
for $x\geq 2$.

Finally, since the proof of (iv) is similar to that of (i), which is
omitted for brevity. This completes the proof.
\end{proof}
By Theorem \ref{thm+sum}, the next result is immediate.
\begin{prop}\label{cor+bino}
Let integers $a,b,c,d,f$ satisfiy $a> c>0$ and $b< d+f+1/2$. Then
the sequence $$\{\sqrt[n]{\frac{\Gamma(an+b+1)}{\Gamma(c
n+d+1)\Gamma((a-c) n+f+1)}}\}_{n\geq1}$$ is asymptotically
log-concave. In particular, $\{\sqrt[n]{\binom{an}{cn}}\}_{n\geq30}$
and $$\{\sqrt[n]{\frac{\Gamma(an+1)}{\Gamma(c n+1)\Gamma((a-c)
n+f+1)}}\}_{n\geq2}$$ is strictly log-concave for $f\geq 1$.
\end{prop}

For integer $p\geq 2$, Fuss-Catalan numbers~\cite{HP} are given by
the formula
$$C_p(n)=\frac{1}{(p-1)n+1}\binom{pn}{n}=\frac{\Gamma(pn+1)}{\Gamma(n+1)\Gamma((p-1)n+2)}.$$
It is well known that the Fuss-Catalan numbers count the number of
paths in the integer lattice $\mathbb{Z}\times \mathbb{Z}$ (with
directed vertices from $(i,j)$ to either $(i,j+1)$ or $(i+1,j))$
from the origin $(0,0)$ to $(n,(p-1)n)$ which never go above the
diagonal $(p-1)x=y$. Su and Wang~\cite{SW08} showed that
$\{\binom{an}{bn}\}_{n\geq 0}$ is log-convex for positive integers
$a>b$. Thus it is easy to see that $\{C_p(n)\}_{n\geq 0}$ is
log-convex. Chen {\it et al.} \cite{CGW1} proved that
$\sqrt[n]{\frac{1}{2n+1}\binom{2n}{n}}$ and
$\sqrt[n]{\binom{2n}{n}}$ are strictly log-concave, respectively. By
verifying the first few terms, one can get the following corollary
by Corollary \ref{cor+bino}.
\begin{cor}
The sequences $\{\sqrt[n]{\frac{1}{2n+1}\binom{2n}{n}}\}_{n\geq1}$,
$\{\sqrt[n]{\binom{2n}{n}}\}_{n\geq1}$,
$\{\sqrt[n]{\binom{3n}{n}}\}_{n\geq1}$,
$\{\sqrt[n]{\binom{4n}{n}}\}_{n\geq1}$,
$\{\sqrt[n]{\binom{5n}{n}}\}_{n\geq1}$ ,
$\{\sqrt[n]{\binom{5n}{2n}}\}_{n\geq1}$ and
$\{\sqrt[n]{C_p(n)}\}_{n\geq2}$ are strictly log-concave for any
positive integer $p\geq2$, respectively.
\end{cor}
\section{Logarithmically completely monotonic functions}
Since logarithmically completely monotonic functions have many
applications, it is important to know which function has such
property. In particular, Chen {\it et al.} \cite{CGW1} found the
connection between logarithmically completely monotonic functions
and infinite log-monotonicity of combinatorial sequences as follows.
\begin{thm}\label{thm+chen}\cite{CGW1}
 Assume that a function $f(x)$ such
that $[\log f(x)]''$ is completely monotonic for $x \geq 1$ and $a_n
= f(n)$ for $n \geq 1$. Then the sequence $\{a_n\}_{n\geq1}$ is
infinitely log-monotonic.
\end{thm}
Thus it is very interesting to research logarithmically complete
monotonicity of some functions related to combinatorial sequences,
which is the aim of this section.

Many sequences of binomial coefficients share various log-behavior
properties, see Tanny and Zuker \cite{TZ74,TZ76}, Su and
Wang~\cite{SW08} for instance. In particular, Su and Wang proved
that $\binom{dn}{\delta n}$ is log-convex in $n$ for positive
integers $d>\delta$. Recently, Chen {\it et al.} \cite{CGW1} proved
that both the Catalan numbers $\frac{1}{n+1}\binom{2n}{n}$ and
central binomial coefficients $\binom{2n}{n}$ are infinitely
log-monotonic. Motivated by these results, a generalization can be
stated as follows.

\begin{thm}\label{thm+e}
Let $n_{0},k_{0},\overline{k_0}$ be nonnegative integers and
$a,b,\overline{b}$ be positive integers. Define the function
\begin{eqnarray*}
G(x)=\frac{\Gamma(n_{0}+ax+1)}{{\Gamma(k_{0}+bx+1)}
{\Gamma(\overline{k_{0}}+x\overline{b}+1)}}.
\end{eqnarray*}
If $a\geq b+\overline{b}$ and $-1\leq k_0-(n_0+1)b/a\leq0$, then
$(\log G(x))''$ is a completely monotonic function for $x\geq0$. In
particular,
$$\{\frac{(n_{0}+ia)!}{(k_0+ib)!(\overline{k_0}+i\overline{b})!}\}_{i\geq0}$$
is infinitely log-monotonic.
\end{thm}
\begin{proof}
By Theorem \ref{thm+chen}, it suffices to show that $(\log G(x))''$
is a completely monotonic function for $x\geq0$. Let $ g(x)=\log
G(x).$ So
\begin{eqnarray}\label{eq}
&&[g(x)]^{(n)}\nonumber\\
&=&[\log\Gamma(n_{0}+ax+1)]^{(n)}-[\log{\Gamma(k_{0}+bx+1)}]^{(n)}-[\log{\Gamma(\overline{k_{0}}+x\overline{b}+1)}]^{(n)}\nonumber\\
&=&(-1)^n\int_{0}^{\infty}\frac{t^{n-1}}{1-e^{-t}}\left[a^ne^{-t(n_{0}+ax+1)}-b^ne^{-t(k_{0}+bx+1)}-\overline{b}^ne^{-t(\overline{k_{0}}+x\overline{b}+1)}\right]dt\\
&=&(-1)^n\int_{0}^{\infty}a^nt^{n-1}e^{-tax}\left[\frac{e^{-(n_0+1)t}}{1-e^{-t}}-\frac{e^{-ta(k_{0}+1)/b}}{1-e^{-at/b}}-\frac{e^{-ta(\overline{k_{0}}+1)/\overline{b}}}{1-e^{-at/\overline{b}}}\right]dt\nonumber
\end{eqnarray}
since
$$[\log\Gamma(x)]^{(n)}=(-1)^{n}\int_{0}^{\infty}\frac{t^{n-1}e^{-tx}}{1-e^{-t}}dt$$ for
$x>0$ and $n\geq2$, see~\cite[p.16]{MOS66} for instance.

It follows from $a>b>0$ that for further simplification denote
$u=k_0-(n_0+1)b/a$, $p=a/b$, and $q=a/\overline{b}$. Clearly,
$\frac{1}{p}+\frac{1}{q}\leq1$. So one can deduce that
\begin{eqnarray}\label{eqq}
(-1)^n[g(x)]^{(n)}
&=&\int_{0}^{\infty}a^nt^{n-1}e^{-t(n_{0}+ax+1)}h(t,u)dt,
\end{eqnarray}
where
$$h(t,u)=\frac{1}{1-e^{-t}}-\frac{e^{-tp(u+1)}}{1-e^{-pt}}-\frac{e^{uqt}}{1-e^{-qt}}.$$

Furthermore, one can obtain the next claim for $-1\leq
k_0-(n_0+1)b/a\leq0$.
\begin{cl}
If $-1\leq u\leq 0$, then $h(t,u)>0$.
\end{cl}
\textbf{Proof of Claim:} It is obvious that $h(t, u)$ is concave in
$u$. Thus it suffices to show $h(t, u)> 0$ for $u =-1$ and $u = 0$.
Setting $u = 0$ since the case $u =-1$ can be obtained by switching
the roles of $p$ and $q$, one has
\begin{eqnarray*}
h(t,0)&=&\frac{e^{-t}}{1-e^{-t}}-\frac{e^{-tp}}{1-e^{-pt}}-\frac{e^{-qt}}{1-e^{-qt}}.
\end{eqnarray*}
Noting for $s>0$ that function $$f(s)=\frac{se^{-s}}{1-e^{-s}}$$
strictly decreases in $s$ and $\frac{1}{p}+\frac{1}{q}\leq1$, one
gets that
\begin{eqnarray*}
h(t,0)&\geq&(\frac{1}{p}+\frac{1}{q})\frac{e^{-t}}{1-e^{-t}}-\frac{e^{-tp}}{1-e^{-pt}}-\frac{e^{-qt}}{1-e^{-qt}}\\
&=&\frac{f(t)-f(tp)}{tp}+\frac{f(t)-f(tq)}{tq}\\
&\ge&0.
\end{eqnarray*}
 This completes the proof of this
Claim.

Thus, by (\ref{eqq}) and this Claim, one has $(-1)^n[g(x)]^{(n)}>0$,
which implies that $(\log G(x))''$ is a completely monotonic
function. This completes the proof.
\end{proof}
By Theorem~\ref{thm+e}, the following two corollaries are immediate.
\begin{cor}
Let $n_{0},k_{0},\D,\d$ be four nonnegative integers. Define the
sequence
\begin{equation*}\label{ai}
    C_i=\binom{n_{0}+i\D}{k_0+i\d},\qquad i=0,1,2,\ldots.
\end{equation*}
If $\D>\d>0$ and $-1\leq k_0-(n_0+1)\d/\D\leq0$, then the sequence
$\{C_n\}_{n\geq0}$ infinitely log-monotonic.
\end{cor}
\begin{cor}
The Fuss-Catalan sequence $\{C_p(n)\}_{n\geq0}$ is  infinitely
log-monotonic, where $p\geq 2$ and
$C_p(n)=\frac{1}{(p-1)n+1}\binom{pn}{n}.$
\end{cor}
The derangements number $d_n$ is a classical combinatorial number.
It is log-convex and ratio log-concave, see \cite{LW07} and
\cite{CGW} respectively. Noted that $\{\Gamma(n)\}_{n\geq1}$ is
strictly infinitely log-monotonic (see Chen {\it et al.}
\cite{CGW1}) and
\begin{eqnarray}\label{e} |d_{n}-\frac{n!}{e}|\leq\frac{1}{2}
\end{eqnarray} for $n\geq 3$ (see \cite{HSW12}),
the following interesting  result can be demonstrated.
\begin{thm}
The sequence of the derangements numbers $\{d_n\}_{n\ge 3}$ is
asymptotically infinitely log-monotonic.
\end{thm}
\begin{proof}
From (\ref{e}), one can deduce
$$\frac{n!}{e}-\frac{1}{2}\leq d_n\leq \frac{n!}{e}+\frac{1}{2},$$
which implies $$\Gamma(n+1)-\frac{3}{2}\leq e d_n\leq
\Gamma(n+1)+\frac{3}{2}.$$ Thus
\begin{eqnarray*}
&&e^2(d_{n+1}d_{n-1}-d_n^2)\\
&\geq&[\Gamma(n+2)-1.5][\Gamma(n)-1.5]-[\Gamma(n+1)+1.5]^2\\
&>0&
\end{eqnarray*}
for $n\ge4$, which implies that $\{d_n\}_{n\ge 4}$ is log-convex.
Note that \begin{eqnarray*}
&&e^4(d_{n+1}^3d_{n-1}-d_{n}^3d_{n+2})\\
&\geq&[\Gamma(n+2)-1.5]^3[\Gamma(n)-1.5]-[\Gamma(n+1)+1.5]^3[\Gamma(n+3)+1.5]\\
&>0&
\end{eqnarray*}
for $n\ge8$, which implies that $R\{d_n\}_{n\geq 8}$ is log-concave.
Because $\{\Gamma(n)\}_{n\geq1}$ is strictly infinitely
log-monotonic, similarly, it can be proceeded to the higher order
log-monotonicity. Thus, for any positive integer $k $, by the
sign-preserving property of limits, one can obtain that there exists
a positive $N$ such that the sequence $R^r\{d_n\}_{n\geq N}$ is
log-concave for positive odd $r$ and is log-convex for positive even
$r$. Thus, the sequence of the derangements numbers $\{d_n\}_{n\ge
3}$ is asymptotically infinitely log-monotonic.
\end{proof}
In the following, we will continous to give two kinds of
logarithmically completely monotonic functions. In order to consider
a stronger result for Theorem~\ref{thm+con}, given $a,b,c>0$, define
the function
$$\theta_{a,b,c}(x)=\sqrt[x]{a \zeta(x+b)\Gamma(x+c)}.$$
It is known that the Riemann zeta function $\zeta(x)$ is
logarithmically completely monotonic on $(1,+\infty)$ and the
function $[\log\Gamma(x)]^{''}$ is completely monotonic on
$(0,+\infty)$, see Chen {\it et al.} \cite{CGW1}. Basing on these
results, one can demonstrate the next.
\begin{thm}
Let $a,b,c$ be positive real numbers, where $b\geq 1$. If
$a\zeta(b)\Gamma(c)\leq1$, then the reciprocal of the function
$\theta_{a,b,c}(x)$ is logarithmically completely monotonic on
$(1,\infty)$.
\end{thm}
\begin{proof}
Since
$$\log\theta^{-1}_{a,b,c}(x)=-\frac{\log\left(a\zeta(x+b)\Gamma(x+c)\right)}{x}=-\frac{\log
a+\log\zeta(x+b)+\log\Gamma(x+c)}{x},$$ in order to show that
$\theta^{-1}_{a,b,c}(x)$ is logarithmically completely monotonic on
$(1,\infty)$, it suffices to prove
$$(-1)^n\log^{(n)}\theta^{-1}_{a,b,c}(x)\geq0$$ for all $n\geq1$.
Note that a known formula as follows:
\begin{eqnarray}\label{eq+d}
\left(\frac{g(x)}{x}\right)^{(n)}=\frac{(-1)^ng(0)n!}{x^{n+1}}+x^{-n-1}\int_0^x
t^n g^{(n+1)}(x)dt,
\end{eqnarray}
which can be easily proved by induction. Thus, one can deduce for
$n\geq1$ and $x>1$ that
\begin{eqnarray*}
&&(-1)^n\log^{(n)}\theta^{-1}_{a,b,c}(x)\\
&&=\frac{-n!\log{a\zeta(b)\Gamma(c)}}{x^{n+1}}+x^{-n-1}\int_0^x t^n
(-1)^{n+1}\left[\left(\log\zeta(x+b)\right)^{(n+1)}+\left(\log\Gamma(x+c)\right)^{(n+1)}\right]dt\\
&&\geq0
\end{eqnarray*}
since $\log{a\zeta(b)\Gamma(c)}\leq0$,
$(-1)^{n+1}\left(\log\zeta(x+b)\right)^{(n+1)}\geq0$ and
$(-1)^{n+1}\left(\log\Gamma(x+c)\right)^{(n+1)}\geq0$.
 This completes the proof.
\end{proof}
In \cite{Al97}, the next result was proved by Alzer.
\begin{thm}\emph{\cite{Al97}}
Let nonnegative sequences $0\leq a_1\leq a_2\leq a_3\leq\cdots\leq
a_n$ and $0\leq b_1\leq b_2\leq b_3\leq\cdots\leq b_n$. If
$\sum_{i=1}^ka_i\leq \sum_{i=1}^kb_i$ for $k=1,2,\ldots,n$, then the
function $$\prod_{i=1}^n\frac{\Gamma(x+a_i)}{\Gamma(x+b_i)}$$ is
 completely monotonic on $(0,\infty)$.
\end{thm}
On the other hand, in \cite{LT11}, Lee and Tepedelenlio\v{g}lu
proved the function
$\sqrt[x]{\frac{2\sqrt{\pi}\Gamma(x+1)}{\Gamma(x+1/2)}}$ originating
from the coding gain is logarithmically completely monotonic on
$(0,\infty)$. In addition, Qi and Li \cite{QL13} considered the
logarithmically completely monotonicity of
$\sqrt[x]{\frac{a\Gamma(x+b)}{\Gamma(x+c)}}$. In what follows a
general result for a kind of logarithmically completely monotonic
functions is obtained.

\begin{thm}\label{thm+comp}
Let $0\leq a_1\leq a_2\leq a_3\leq\cdots\leq a_n$, $0\leq b_1\leq
b_2\leq b_3\leq\cdots\leq b_n$ and $\rho>0$, define the function
$$\chi(x)=\sqrt[x]{\rho\prod_{i=1}^n\frac{\Gamma(x+a_i)}{\Gamma(x+b_i)}}.$$
\begin{itemize}
\item [\rm (i)]
If $\rho\prod_{i=1}^n\frac{\Gamma(a_i)}{\Gamma(b_i)}\geq1$ and
$\sum_{i=1}^ka_i\geq \sum_{i=1}^kb_i$ for $k=1,2,\ldots,n$, then the
function $\chi(x)$ is logarithmically completely monotonic on
$(0,\infty)$.
\item [\rm (ii)]
If $\rho\prod_{i=1}^n\frac{\Gamma(a_i)}{\Gamma(b_i)}\leq1$ and
$\sum_{i=1}^ka_i\leq \sum_{i=1}^kb_i$ for $k=1,2,\ldots,n$, then the
reciprocal of the function $\chi(x)$ is logarithmically completely
monotonic on $(0,\infty)$.
\end{itemize}
\end{thm}

\begin{proof} Because (ii) can be obtained in the similar
way, it only needs to prove (i). Define the function
$h(x)=\sum_{i=1}^n\log \Gamma(x+a_i)-\log\Gamma(x+b_i)$. Then
$$\log\sqrt[x]{\rho\prod_{i=1}^n\frac{\Gamma(x+a_i)}{\Gamma(x+b_i)}}=\frac{\log\rho\prod_{i=1}^n\frac{\Gamma(x+a_i)}{\Gamma(x+b_i)}}{x}
=\frac{\log\rho+h(x)}{x}.$$ So it is not hard to get
\begin{eqnarray}\label{muti}
&&(-1)^k[\log\chi(x)]^{(k)}=\frac{k!(\log\rho+h(0))}{x^{k+1}}+x^{-k-1}\int_0^x
t^k (-1)^kh^{(k+1)}(x)dt.
\end{eqnarray}

If $\rho\prod_{i=1}^n\frac{\Gamma(a_i)}{\Gamma(b_i)}\geq1$, then it
is clear that $$\log\rho+h(0)\geq0.$$ In addition, Alzer \cite{Al97}
proved that $(-1)^kh^{(k+1)}(x)\geq0$ for $k\geq0$ and $x\geq0$.
Thus, $$(-1)^k[\log\chi(x)]^{(k)}\geq0,$$ that is, $\chi(x)$ is
logarithmically completely monotonic on $(0,\infty)$. This completes
the proof.
\end{proof}
\begin{rem}
If $\rho=2\sqrt{\pi}$, $a_1=1$ and $b_1=\frac{1}{2}$, then
$\frac{2\sqrt{\pi}\Gamma(1)}{\Gamma(1/2)}=2>1$. So the function
$\sqrt[x]{\frac{2\sqrt{\pi}\Gamma(x+1)}{\Gamma(x+1/2)}}$ is
logarithmically completely monotonic on $(0,\infty)$, see
\cite{LT11}. In addition, if $n=1$ in Theorem~\ref{thm+comp}, then
it was proved by Qi and Li \cite{QL13}. Thus, the result in
Theorem~\ref{thm+comp} is a generalization.
\end{rem}
\section{Acknowledgements}
The author would like to thank the anonymous reviewer for many
valuable remarks and suggestions to improve the original manuscript.
He also wishes to thank his advisor Prof. Yi Wang at Dalian
University of Technology for his advice, support, and constant
encouragement during the course of his research.



\end{document}